\documentclass[12pt,oneside,reqno]{amsart}

\usepackage[top=2.5cm,bottom=2.5cm,left=3cm,right=3cm]{geometry}

\usepackage{microtype}
\usepackage{mathtools}
\usepackage{mathrsfs}
\usepackage{newtxtext,newtxmath}
\usepackage[mathcal]{eucal}
\usepackage{tikz-cd}
\usepackage{enumitem}
\setlist{itemsep=1pt,topsep=4pt}
\usepackage{graphicx}
\usepackage{comment}
\usepackage{amsthm}
\usepackage{xcolor}
\definecolor{citeblue}{rgb}{0,0.35,0.75}

\theoremstyle{plain}
\newtheorem{theorem}{Theorem}[section]
\newtheorem{proposition}[theorem]{Proposition}
\newtheorem{lemma}[theorem]{Lemma}
\newtheorem{corollary}[theorem]{Corollary}

\theoremstyle{definition}
\newtheorem{definition}[theorem]{Definition}
\newtheorem{remark}[theorem]{Remark}
\newtheorem{question}[theorem]{Question}

\numberwithin{equation}{section}

\usepackage[
  colorlinks=true,
  linkcolor=citeblue,
  citecolor=citeblue,
  urlcolor=citeblue
]{hyperref}
\usepackage[nameinlink,noabbrev]{cleveref}

\crefname{equation}{}{}
\Crefname{section}{\S}{}
\Crefname{question}{Question}{Questions}
\Crefname{theorem}{Theorem}{Theorems}

\newcommand{\bfx}{\mathbf{x}}

\DeclareMathOperator{\GL}{GL}
\DeclareMathOperator{\SC}{SC}
\DeclareMathOperator{\Heg}{Heg}

\newcommand{\cR}{\mathcal{R}}
\newcommand{\cS}{\mathcal{S}}

\newcommand{\rO}{\mathrm{O}}

\newcommand{\CC}{\mathbb{C}}

\newcommand{\PP}{\mathbb{P}}
\newcommand{\RR}{\mathbb{R}}

\newcommand{\QQ}{\mathbb{Q}}
\newcommand{\ZZ}{\mathbb{Z}}

\renewcommand{\AA}{\mathbb{A}}

\newcommand{\BB}{\mathrm{BB}}

\DeclareMathOperator{\CH}{CH}

\DeclareMathOperator{\Pic}{Pic}
\DeclareMathOperator{\Sym}{Sym}
\DeclareMathOperator{\Span}{Span}
\DeclareMathOperator{\Supp}{Supp\,}

\newcommand{\Cl}{\operatorname{Cl}}

\title[Special cycles and Picard groups]{Special cycles and Picard groups of Baily--Borel compactifications}
\author{Zhiyuan Li}
\author{Zelin Ye}
\date{\today}

\address{Shanghai Center for Mathematical Sciences, Fudan University, 2005 Songhu Road, 200438 Shanghai, China}
\email{zhiyuan\underline{ }li@fudan.edu.cn}

\address{Shanghai Center for Mathematical Sciences, Fudan University, 2005 Songhu Road, 200438 Shanghai, China}
\email{zlye24@m.fudan.edu.cn}

\begin{document}
\begin{abstract}
In this paper, we answer two basic questions concerning algebraic cycles on orthogonal Shimura varieties attached to lattices. First, we prove that Kudla's special cycles and Heegner cycles have the same rational span in every codimension. This fills a gap in the codimension-one argument of \cite[Remark 2.4]{BLMM17}, pointed out by Bruinier--Zuffetti \cite{BZ26}. Second, in dimension at least ten, we prove that the rational Picard group of the Baily--Borel compactification is generated by the Hodge class. The proof of the second result reduces, via isotropic induction on the discriminant form, to the K3-type case treated in our previous work.
\end{abstract}

\maketitle

\section{Introduction}

Special cycles on orthogonal Shimura varieties play a central role in Kudla's program, which relates their generating series and arithmetic intersection numbers to modular forms and derivatives of Eisenstein series \cite{Kud04}. The work of Kudla--Millson \cite{KM90} relates their cohomology classes to Siegel modular forms through theta series. In the adelic formulation, special cycles are defined using finite Schwartz functions that specify arithmetic conditions on the defining vectors.

A distinguished family is given by Heegner cycles defined on orthogonal Shimura varieties $X_L$ attached to a lattice $L$, indexed by Gram matrices and discriminant classes of the lattice. In codimension one, they support the divisors of Borcherds products \cite{Bor98}. They also arise naturally in moduli theory: on moduli spaces of quasi-polarized K3 surfaces, Heegner divisors describe the Noether--Lefschetz conditions studied by Maulik--Pandharipande \cite{MP13}. In the range covered by \cite{BLMM17}, these divisors generate the rational Picard group (see also \cite{BZ26}). Under the hypotheses of \cite{HLMY26}, the classes that extend to the Baily--Borel compactification are precisely the rational multiples of the Hodge class.

This suggests two basic questions for orthogonal Shimura varieties:
\begin{enumerate}[label=\textup{(\Roman*)}]
    \item On $X_L$, is the span of Kudla's special cycles generated by Heegner cycles? Moreover, does it contain the classes of all their irreducible components?

    \item For orthogonal Shimura varieties, which special divisor classes extend to the Baily--Borel compactification?
\end{enumerate}

We restrict Question~(I) to $X_L$, since the span of special cycles may grow at deeper congruence levels, while the dimension of the Heegner span remains unchanged. Question~(I) is also relevant to the tautological rings of many moduli spaces. For example, irreducible Noether--Lefschetz cycle classes generate the tautological ring of the moduli space of polarized hyperk\"ahler varieties. Understanding whether these classes lie in the Heegner span would therefore connect tautological rings with modular forms and imply their finite generation.

Question~(II) makes sense for arbitrary congruence quotients, but a  solution will follow from the case of $X_L$ as $L$ varies.

In this paper, we resolve the first part of Question~(I) and Question~(II)  in substantial generality.

\subsection{Span of special cycles and Heegner cycles}

Let $L$ be a nondegenerate even lattice of signature $(2,n)$, put $V=L\otimes\QQ$, and write $A_L=L^\vee/L$. Fix a connected component $\mathcal D^+$ of the associated type IV domain, and set
\[
\widetilde{\rO}^{+}(L)=\ker\bigl(\rO^+(L)\longrightarrow\rO(A_L)\bigr), \qquad X_L=\widetilde{\rO}^{+}(L)\backslash\mathcal D^+.
\]
Let $K_L\subset\rO(V)(\AA_f)$ be the compact open subgroup preserving $\widehat L$ and acting trivially on $A_L$.

For $1\leq r\leq n$, a negative definite matrix $T\in\Sym_r(\QQ)$, and $\varphi\in\cS(V(\AA_f)^r,\QQ)^{K_L}$, the adelic special cycle $Z(T,\varphi)$ is defined by tuples $\bfx$ satisfying
\[
q(\bfx)=\tfrac12((x_i,x_j))=T,
\]
weighted by $\varphi(\bfx)$. Taking $\varphi=\mathbf1_{\mu+\widehat L^r}$, for $\mu\in A_L^r$, gives the Heegner cycle
\[
Z_L(T,\mu)=Z(T,\mathbf1_{\mu+\widehat L^r}).
\]
The cycle $Z_L(T,\mu)$ includes all tuples with the prescribed Gram matrix and discriminant data and need not be irreducible.

\begin{remark}
    Heegner cycles are also defined for congruence subgroups of $\widetilde{\rO}^{+}(L)$. Passing to such a subgroup preserves the dimension of their span \cite[(2.12)]{Kud04}, whereas the dimension of the span of Kudla's special cycles can increase (see, for example, \cite[\S~3]{Bru02}). The comparison below is therefore restricted to the full stable orthogonal group $\widetilde{\rO}^{+}(L)$.   
\end{remark}

Let $\SC^r(X_L)$ and $\Heg^r(X_L)$ denote the rational spans of their classes in $\CH^r(X_L)_\QQ$. We will in fact prove the comparison before passing to rational equivalence. There is a natural inclusion
\[
\Heg^r(X_L)\subseteq\SC^r(X_L).
\]
The reverse inclusion in codimension one was asserted in \cite[Remark~2.4]{BLMM17}. As pointed out by Bruinier--Zuffetti \cite[Remark~4.9]{BZ26}, the argument there is incomplete: a $K_L$-invariant Schwartz function need not be a linear combination of characteristic functions of cosets of $\widehat L$. Our first theorem fills this gap and establishes the equality of the two spans in every codimension.

\begin{theorem}\label{thm:comparison-main}
    Let $L$ be an even lattice of signature $(2,n)$. For every $1\leq r\leq n$, we have 
    \[
    \SC^r(X_L)=\Heg^r(X_L),
    \]
    in $\CH^r(X_L)_\QQ$.
\end{theorem}

The theorem allows generation results for special cycles to be
restated in terms of Heegner cycles. The proof combines a local
extension theorem for saturated embeddings with decomposition
by saturation: saturated tuples with the same Gram matrix and
discriminant classes lie in the same $K_L$-orbit, so every
$K_L$-invariant Schwartz function takes the same value on them.

A stronger question concerns the individual irreducible components of these cycles.

\begin{question}
    Does the span of Heegner cycles contain the classes of all their irreducible components?
\end{question}

This certainly holds when primitive Heegner cycles are irreducible. Such irreducibility, however, generally requires additional assumptions on the lattice. The question remains open in general: Theorem~\ref{thm:comparison-main} identifies the spans of two families of weighted sums, but does not separate their individual components.

\subsection{Picard group of Baily--Borel compactifications}

We next consider the Picard group of Baily--Borel compactification of orthogonal Shimura varieties. For a congruence subgroup $\Gamma\subseteq\rO^+(L)$, write
\[
X_\Gamma=\Gamma\backslash\mathcal D^+,\qquad \overline X_\Gamma=\overline{X}_\Gamma^{\BB}.
\]
Let $\lambda_\Gamma$ be the rational Hodge class and $\bar\lambda_\Gamma$ its extension, with the sign convention for which $\bar\lambda_\Gamma$ is ample.

The boundary of $\overline X_\Gamma$ has dimension at most one \cite{BB66}. Nevertheless, a divisor on $X_\Gamma$ need not have a $\QQ$-Cartier closure; see \cite[Proposition~5.1]{BF01} for a local obstruction. Thus generation of the open Picard group by special divisors leaves a separate problem of determining which combinations extend. In our earlier work \cite{HLMY26}, we studied this obstruction and proved Picard rank one under arithmetic hypotheses on the lattice. In particular, our result applies at the stable orthogonal level to lattices of K3 type, which split $2U$ over $\ZZ$ and $3U$ over every $\ZZ_p$. We remove these splitting hypotheses in dimension at least ten.

\begin{theorem}\label{thm:picard-main}
    Let $L$ be an even lattice of signature $(2,n)$ with $n\geq10$. For every congruence subgroup $\Gamma\subseteq\rO^+(L)$,
    \[
    \Pic(\overline X_\Gamma)_\QQ=\QQ\bar\lambda_\Gamma.
    \]
\end{theorem}

Since restriction to $X_\Gamma$ is injective on rational Picard groups, the theorem says that a rational divisor class extends precisely when it is a multiple of the Hodge class. It therefore determines the extendable part of the open Picard group uniformly in the lattice and congruence level.

\medskip
\subsection{Idea of the proof}
\begin{itemize}
    \item For Theorem~\ref{thm:comparison-main}, let $\Heg_{\mathrm{prim}}^r(X_L)$ denote the span of primitive Heegner-cycle classes and let $\mathrm{Irr}^r(X_L)$ denote the span of the classes of irreducible components of special cycles. We have
    \[
    \mathrm{Heg}_{\mathrm{prim}}^r(X_L)=\mathrm{Heg}^r(X_L) \subseteq \mathrm{SC}^r(X_L) \subseteq \mathrm{Irr}^r(X_L).
    \]
    The first equality follows from the decomposition and inversion formulas of Bruinier–Zuffetti \cite[Lemma~4.13]{BZ26}. We actually prove the reverse inclusion
    \[
    \mathrm{SC}^r(X_L)\subseteq \mathrm{Heg}_{\mathrm{prim}}^r(X_L).
    \]

    The key is Lemma~\ref{lem:local-orbits}: saturated tuples with the same Gram matrix and discriminant classes lie in the same $K_L$-orbit and hence have the same Schwartz weight. Decomposition by saturation therefore expresses every special cycle as a finite rational combination of primitive Heegner cycles.

    \item For Theorem~\ref{thm:picard-main}, the key reduction is achieved by \emph{isotropic induction} on the discriminant form. Starting from $A_L$, we pass to even overlattices $M=L_H\subset L^\vee$ associated with isotropic subgroups $H\subset A_L$. The induced finite morphisms $f_H: X_L\to X_M$ pull Heegner divisors back according to the coset decomposition of $M$ over $L$, and the corresponding isotropic induction maps $\uparrow_H^{A_L}$ span the relevant part of $\QQ[A_L]$. Iterating this procedure, we reduce to finitely many overlattices $M_i$ of K3 type, to which the known Picard rank one results apply. Section~\ref{sec:auxiliary} constructs these overlattices (Corollary~\ref{cor:auxiliary-overlattices}). Proposition~\ref{prop:open-Picard-coverage} then shows that $\Pic(X_L)_\QQ$ is generated by pullbacks of rational Picard classes on the $X_{M_i}$. Proposition~\ref{prop:joint} then combines this generation statement with finite norms and the Hodge index theorem to deduce Picard rank one for $\overline X_L$ from the corresponding results for the $\overline X_{M_i}$.

    To pass to a general congruence quotient \(X_\Gamma\), we choose a finite cover \(X_{aL}\to X_\Gamma\). A line bundle on \(\overline{X}_\Gamma\) is determined, up to torsion, by its pullback to \(\overline{X}_{aL}\). 

\end{itemize}

\subsection{Organization of the paper}
Section~\ref{sec:prelim} fixes the lattice, quotient, and cycle conventions. Section~\ref{sec:comparison} develops the saturation argument and proves the cycle comparison. Section~\ref{sec:auxiliary} constructs the overlattices needed for the Picard theorem. Section~\ref{sec:picard} first proves generation by pullback, then uses finite norms and the Hodge index theorem to complete the proof of Theorem~\ref{thm:picard-main}. 

\subsection*{Origins of the work and disclosure of AI use}

The first question addressed in this paper arose from a letter from Bruinier, who asked whether the gap in the codimension-one argument of \cite[Remark~2.4]{BLMM17} could be filled and whether the equality of spans could be established in every codimension. This led to the comparison theorem proved in Section~\ref{sec:comparison}.

The second question continues our work in \cite{HLMY26}. The idea of using isotropic subgroups of the discriminant form to reduce questions about general lattices to the K3-type case was discussed extensively during that earlier work.

LLM was used in the preparation of this paper. The LLM assisted with the local orbit criterion in Lemma~\ref{lem:local-orbits}, a local analogue of Eichler's criterion. It also suggested the key idea of using a finite family of overlattices whose pullbacks jointly span the rational Picard group, rather than seeking a reduction through a single lattice. This idea is implemented in Proposition~\ref{prop:open-Picard-coverage}; the resulting spanning statement is then used with Proposition~\ref{prop:joint} and our earlier result in \cite{HLMY26}.

The LLM was also used as a writing aid to improve the language and readability of the manuscript. The authors have reviewed and verified all arguments and take full responsibility for the content of the paper.

\subsection*{Conventions}All varieties are over $\CC$; Picard and Chow groups are taken with rational coefficients.

\subsection*{Acknowledgements}The author is grateful to Hanyu Yao for carefully reading an earlier version of this paper and providing useful comments. Zhiyuan Li is supported by NSFC grant 12425105 and Shanghai Pilot Program for Basic Research (No. 21TQ001).

\section{Preliminary}\label{sec:prelim}

We fix the discriminant-form notation and the conventions for cycles on coarse arithmetic quotients. The normalization under change of level will be used in both the comparison theorem and the pullback formula for Heegner divisors.

\subsection{Discriminant forms and even overlattices}

Unless another base ring is specified, a lattice is a finite free $\ZZ$-module $L$ equipped with a nondegenerate symmetric bilinear form $(\, ,\,):L\times L\to\ZZ$ such that $(x,x)\in2\ZZ$ for every $x\in L$. We write
\[
q(x)=\frac{(x,x)}2,\qquad L_{\QQ}=L\otimes_{\ZZ}\QQ,\qquad L_{\RR}=L\otimes_{\ZZ}\RR,
\]
and use the same notation for the extensions of the bilinear and quadratic forms. The signature of $L$ is denoted by $\operatorname{sign}(L)=(b^+(L),b^-(L))$. For each prime $p$, put
\[
L_p=L\otimes_{\ZZ}\ZZ_p,\qquad \widehat L=L\otimes_{\ZZ}\widehat\ZZ.
\]

The dual lattice and discriminant group are
\[
L^\vee=\{x\in L_{\QQ}:(x,L)\subseteq\ZZ\}, \qquad A_L=L^\vee/L.
\]
The finite abelian group $A_L$ carries the discriminant quadratic form and its associated bilinear pairing
\[
\bar{q}_L(\bar x)=q(x)\bmod\ZZ,\qquad \bar{b}_L(\bar x,\bar y)=(x,y)\bmod\ZZ,
\]
where $\bar x=x+L$ and $\bar y=y+L$. We denote by $\rO(A_L)$ the group of automorphisms of $A_L$ preserving $\bar{q}_L$. The \emph{stable orthogonal group}, or \emph{discriminant kernel},
is
\[
\widetilde{\rO}(L)=\ker\bigl(\rO(L)\longrightarrow\rO(A_L)\bigr).
\]

\begin{definition}
    A submodule $N$ of a finite free module $P$ is called \emph{saturated} if $P/N$ is torsion-free. Over $\ZZ$ or a discrete valuation ring this is equivalent to $N$ being a module direct summand of $P$. A nonzero vector is \emph{primitive} if the rank-one submodule it generates is saturated. Saturation and primitivity will always refer to the specified ambient module.
\end{definition}

\begin{definition}

A subgroup $H\subseteq A_L$ is \emph{isotropic} if $\bar{q}_L|_H=0$. Set
\[
H^\perp=\{\alpha\in A_L:\bar{b}_L(\alpha,H)=0\}.
\]
\end{definition}
The assignment
\[
H\longmapsto L_H:=\{x\in L^\vee:x+L\in H\}
\]
gives a bijection between isotropic subgroups of $A_L$ and even overlattices of $L$ contained in $L^\vee$. Moreover,
\begin{equation}\label{eq:overlattice-discriminant}
    L_H/L=H,\qquad L_H^\vee/L=H^\perp,\qquad (A_{L_H},\bar{q}_{L_H})\simeq(H^\perp/H,\bar q_{L,H}),
\end{equation}
where $\bar q_{L,H}$ denotes the quadratic form induced by $\bar{q}_L$; see \cite[Proposition~1.4.1]{Ni79}. We regard $\widetilde{\rO}(L)$ and $\widetilde{\rO}(L_H)$ as subgroups of $\rO(L_{\QQ})$. Every element of $\widetilde{\rO}(L)$ preserves $L_H$ and acts trivially on $H^\perp/H$. Hence
\begin{equation}\label{eq:overlattice-stable-inclusion}
    \widetilde{\rO}(L)\subseteq\widetilde{\rO}(L_H).
\end{equation}

\begin{definition}
    Let $U$ be the even unimodular hyperbolic plane, and put $U_p=U\otimes\ZZ_p$. We say that a lattice of signature $(2,n)$ is \emph{of K3 type} if it splits $U^{\oplus2}$ over $\ZZ$ and $U_p^{\oplus3}$ over $\ZZ_p$ for every prime $p$. This is the splitting condition used in the compactified Picard theorem recalled in Section~\ref{subsec:picard-inputs}.
\end{definition}

\subsection{Orthogonal Shimura varieties}

Let $L$ have signature $(2,n)$, put $V=L\otimes\QQ$, and write $G=\rO(V)$. The type~IV domain
\[
\mathcal D=\{[z]\in\PP(V\otimes\CC):(z,z)=0,\ (z,\bar z)>0\}
\]
has two connected components, each of dimension $n$. Fix a component $\mathcal D^+$ and let $G(\QQ)^+$ be its stabilizer. Set
\[
\rO^+(L)=\rO(L)\cap G(\QQ)^+,\qquad \widetilde{\rO}^{+}(L)=\ker\bigl(\rO^+(L)\longrightarrow\rO(A_L)\bigr),\qquad X_L=\widetilde{\rO}^{+}(L)\backslash\mathcal D^+.
\]
The variety $X_L$ is normal and quasi-projective, with finite quotient singularities. In particular, $X_L$ is $\mathbb Q$-factorial. Its Baily--Borel compactification $\overline X_L$ is normal and projective \cite{BB66}. The boundary is a finite union of points and modular curves, indexed respectively by $\widetilde{\rO}^{+}(L)$-orbits of rational isotropic lines and planes in $V$ \cite[Theorem~5.11]{GHS13}. Hence
\[
\dim(\overline X_L\setminus X_L)\leq1.
\]
The Hodge class $\lambda_L\in\Pic(X_L)_\QQ$ extends to a class $\bar\lambda_L\in\Pic(\overline X_L)_\QQ$; we choose the sign for which $\bar\lambda_L$ is ample. We use the same notation for their first Chern classes.

Let $\AA_f$ denote the finite adeles of $\QQ$. The compact open subgroup associated with $L$ is
\[
K_L=\prod_pK_{L,p},\qquad K_{L,p}=\ker\bigl(\rO(L_p)\longrightarrow\rO(L_p^\vee/L_p)\bigr).
\]
Thus $K_L$ preserves $\widehat L$ and acts trivially on $\widehat{L^\vee}/\widehat L$, and $\widetilde{\rO}^{+}(L)=G(\QQ)^+\cap K_L$. In the adelic description
\begin{equation}\label{eq:adelic-shimura}
    \operatorname{Sh}_K(V)(\CC)=G(\QQ)\backslash\bigl(\mathcal D\times G(\AA_f)/K\bigr),
\end{equation}
$X_L$ is the component represented by $1\in G(\AA_f)$ when $K=K_L$. More generally, for a compact open subgroup $K\subset G(\AA_f)$, write
\begin{equation}\label{eq:adelic-components}
    \Gamma=G(\QQ)^+\cap K,\qquad X_\Gamma=\Gamma\backslash\mathcal D^+.
\end{equation}
We denote its compactification and Hodge classes by $\overline X_\Gamma$, $\lambda_\Gamma$, and $\bar\lambda_\Gamma$. All quotients are understood as coarse varieties, and Picard and Chow groups have rational coefficients.

\subsection{Special cycles and Heegner cycles}\label{subsec:special-heegner}

For $1\leq r\leq n$ and a row tuple $\bfx=(x_1,\ldots,x_r)\in V^r$, put
\[
q(\bfx)=\tfrac12((x_i,x_j))_{i,j}.
\]
For a negative definite matrix $T\in\Sym_r(\QQ)$, let
\[
\Omega_T=\{\bfx\in V^r:q(\bfx)=T\},\qquad \mathcal D^+_{\bfx}=\{[z]\in\mathcal D^+:(z,x_i)=0\text{ for all }i\}.
\]
Each $\mathcal D^+_{\bfx}$ has codimension $r$. If $\Gamma_{\bfx}\subset\Gamma$ is the pointwise stabilizer of $\bfx$, the natural map
\[
\Gamma_{\bfx}\backslash\mathcal D^+_{\bfx}\longrightarrow X_\Gamma
\]
defines a cycle $Z_\Gamma(\bfx)$ by pushforward, with the normalization for non-neat groups specified below.

Let $\cS(V(\AA_f)^r,\QQ)$ denote the space of compactly supported, locally constant rational-valued functions. For a $K$-invariant function $\varphi$ in this space, the associated \emph{special cycle} on $X_\Gamma$ is
\begin{equation}\label{eq:special-cycle-component}
    Z(T,\varphi;K)|_{X_\Gamma}=\sum_{\bfx\in\Gamma\backslash\Omega_T}\varphi(\bfx)\,Z_\Gamma(\bfx).
\end{equation}
Only finitely many terms are nonzero; see \cite[\S~2]{KM90} and \cite[Section~2.6]{BLMM17}.

For $\mu\in A_L^r$, the coset characteristic function $\varphi_\mu=\mathbf1_{\mu+\widehat L^r}$ is $K_L$-invariant. The corresponding \emph{Heegner cycle} is
\begin{equation}\label{eq:special-as-adelic}
    Z_L(T,\mu)=Z(T,\varphi_\mu;K_L)|_{X_L}.
\end{equation}
The coset is independent of the lift of $\mu$ to $(L^\vee)^r$, and its rational points are $\mu+L^r$. For $r=1$, we write $Z_L(t,\mu)$, where $t<0$; this divisor vanishes unless $t\bmod\ZZ=\bar{q}_L(\mu)$.

For completeness, we specify the multiplicities used above. When $\Gamma$ is neat, $Z_\Gamma(\bfx)$ is the pushforward of the fundamental cycle, including the degree onto its image. In general, choose a neat normal subgroup $\Gamma'\subset\Gamma$ of finite index, with quotient map $p:X_{\Gamma'}\to X_\Gamma$, and define
\[
Z_\Gamma(\bfx)=\frac{1}{\deg p}\,p_*\left(\sum_{\gamma\in\Gamma'\backslash\Gamma/\Gamma_{\bfx}}Z_{\Gamma'}(\gamma\bfx)\right).
\]
Here $\deg p$ is the degree of the map of coarse varieties. This definition is independent of $\Gamma'$ \cite[(2.12)]{Kud04}. For special divisors, these multiplicities give the pullback identity
\begin{equation}\label{eq:special-divisor-pullback}
    f^*\bigl(Z(t,\varphi;K)|_{X_\Gamma}\bigr)=Z(t,\varphi;K')|_{X_{\Gamma'}},
\end{equation}
where $K'\subset K$, $\Gamma'=G(\QQ)^+\cap K'$, and $f:X_{\Gamma'}\to X_\Gamma$ is the induced finite map.

When $L$ is fixed, we abbreviate
\[
Z(T,\varphi)=Z(T,\varphi;K_L)|_{X_L},\qquad Z(T,\mu)=Z_L(T,\mu).
\]
Let $\SC^r(X_L)$ and $\Heg^r(X_L)$ be the rational spans of special cycles and Heegner cycles, respectively, in $\CH^r(X_L)_\QQ$. Then
\[
\Heg^r(X_L)\subseteq\SC^r(X_L).
\]
We prove the reverse inclusion in Section~\ref{sec:comparison}.

\section{Schwartz functions and cycle spans}\label{sec:comparison}

In this section, we show that cycles defined by arbitrary $K_L$-invariant Schwartz functions have the same rational span as those defined by the characteristic functions $\mathbf1_{\mu+\widehat L^r}$, indexed by $\mu\in A_L^r$. This is the key ingredient for proving Theorem \ref{thm:comparison-main}.

\subsection{Local orbits of saturated tuples}

Let $R$ be a discrete valuation ring with fraction field $F$ of characteristic different from $2$. An \emph{integral quadratic $R$-lattice} is a finite free $R$-module $L$ with a quadratic form $q:L\to R$ whose polar form
\[
(x,y)=q(x+y)-q(x)-q(y)
\]
is nondegenerate on $L_F=L\otimes_R F$. Put
\[
L^\vee=\{x\in L_F:(x,L)\subseteq R\},\qquad A_L=L^\vee/L,\qquad \bar{q}_L(x+L)=q(x)\bmod R,
\]
and
\[
\widetilde{\rO}(L)=\ker\bigl(\rO(L)\longrightarrow\rO(A_L,\bar{q}_L)\bigr).
\]
The induced discriminant form $\bar{q}_L$ takes values in $F/R$.

The following lemma characterizes the $\widetilde{\rO}(L)$-orbits of embeddings into $L^\vee$ with saturated image.

\begin{lemma}\label{lem:local-orbits}
    Let $(L,q)$ be an integral quadratic $R$-lattice, and let
    \[
    X,Y:R^r\hookrightarrow L^\vee
    \]
    be embeddings with saturated images in $L^\vee$. There exists $g\in\widetilde{\rO}(L)$ satisfying $g\circ X=Y$ if and only if
    \begin{equation}\label{eq:local-orbit-invariants}
        q(Xa)=q(Ya) \quad\text{and}\quad Xa-Ya\in L \qquad\text{for every }a\in R^r.
    \end{equation}
\end{lemma}
\begin{proof}
    Necessity is immediate. Conversely, assume \eqref{eq:local-orbit-invariants}. Put $S=X(R^r)$, $T=Y(R^r)$, and define
    \begin{align*}
        \sigma: S & \to T \\
        Xa & \mapsto Ya.
    \end{align*} 
    The assumed identities imply that $\sigma$ is an isometry and that $\sigma(s)-s\in L$ for every $s\in S$. Since $S$ and $T$ are saturated in $L^\vee$, the pairing induces surjections
    \[
    L\xrightarrow{\sim}\operatorname{Hom}_R(L^\vee,R) \twoheadrightarrow\operatorname{Hom}_R(E,R), \qquad E\in\{S,T\}.
    \] 
    Kneser’s extension theorem \cite[Theorem~1.2.2]{Kit93}, applied with 
    \[
    (\widetilde R,U,L,M,N,H)=(F,L_F,L^\vee,S,T,L)
    \]
    in Kitaoka’s notation, gives an isometry $g$ of $L^\vee$ extending $\sigma$ with $(g-1)L^\vee\subseteq L$. Since $g$ also preserves $(L^\vee)^\vee=L$, it belongs to $\widetilde{\mathrm O}(L)$.
\end{proof}

\subsection{Primitive Heegner cycles}

Recall that an independent tuple $\mathbf y=(y_1,\ldots,y_r)\in(L^{\vee})^r$ is \emph{saturated} if
\[
\ZZ y_1+\cdots+\ZZ y_r=L^{\vee}\cap\bigl(\QQ y_1+\cdots+\QQ y_r\bigr).
\]
Equivalently, the associated embedding $\ZZ^r\hookrightarrow L^{\vee}$ has torsion-free cokernel. For $\nu\in A_L^r$, write
\[
\mathcal P_r(L,\nu)=\bigl\{\mathbf y\in\nu+L^r:\mathbf y\text{ is an independent saturated tuple in }L^{\vee}\bigr\}.
\]

We use the cycles $Z_{\widetilde{\rO}^{+}(L)}(\mathbf y)$ defined in Section~\ref{subsec:special-heegner}.

\begin{definition}[{\cite[Definition~4.10]{BZ26}}]\label{def:primitive-special-cycles}
    For $S\in\operatorname{Sym}_r(\QQ)_{<0}$ and $\nu\in A_L^r$, the \emph{primitive Heegner cycle} is
    \[
    Z^{\mathrm{prim}}(S,\nu)=\sum_{\substack{\mathbf y\in\widetilde{\rO}^{+}(L)\backslash\mathcal P_r(L,\nu)\\
    q(\mathbf y)=S}}Z_{\widetilde{\rO}^{+}(L)}(\mathbf y).
    \]
\end{definition}

\begin{remark}
    In general, primitivity does not imply irreducibility. The irreducible components of the support of $Z^{\mathrm{prim}}(S,\nu)$ correspond to the $\widetilde{\rO}^{+}(L)$-orbits of the rational subspaces $\Span_\QQ(y_1,\ldots,y_r)$ occurring in its definition. In particular, the support is irreducible if the defining tuples form a single $\widetilde{\rO}^{+}(L)$-orbit.

    For $r=1$, suppose that $L$ splits $U^{\oplus2}$. By the Eichler criterion \cite[Proposition~3.3(i)]{GHS09}, primitive vectors in $L^\vee$ with the same norm and discriminant class form a single $\widetilde{\rO}^{+}(L)$-orbit. Thus every nonzero primitive Heegner divisor has irreducible support. For $r=2$, an analogous irreducibility result holds under a stronger splitting hypothesis; see \cite[Proposition~4.15]{BZ26}. 
\end{remark}

The following is a direct consequence of \cite[Lemma~4.13]{BZ26}.

\begin{lemma}\label{lem:finite-saturation-inversion}
    For $1\leq r\leq n$, Heegner cycles and primitive Heegner cycles span the same rational subspace of codimension-$r$ algebraic cycles on $X_L$.
\end{lemma}

\subsection{The comparison theorem}

We start with a simple result:

\begin{lemma}\label{lem:schwartz-constant-on-saturated-tuples}
    Let $S\in\Sym_r(\QQ)_{<0}$, $\nu\in A_L^r$, and $C\in M_r(\QQ)$. For every $\varphi\in\cS(V(\AA_f)^r,\QQ)^{K_L}$, the function $\mathbf y\mapsto\varphi(\mathbf yC)$ is constant on
    \[
    \{\mathbf y\in\mathcal P_r(L,\nu):q(\mathbf y)=S\}.
    \]
\end{lemma}
\begin{proof}
    Let $\mathbf y,\mathbf y'$ belong to this set. At every prime $p$, they define saturated embeddings into $L_p^\vee$ with the same Gram matrix and discriminant classes. Lemma~\ref{lem:local-orbits} gives $k_p\in K_{L,p}$ with $k_p\mathbf y=\mathbf y'$. Hence $k=(k_p)_p\in K_L$ satisfies
    \[
    \varphi(\mathbf y'C)=\varphi\bigl(k(\mathbf yC)\bigr)=\varphi(\mathbf yC).
    \]
\end{proof}

Combining this fact with Lemma~\ref{lem:finite-saturation-inversion}, we obtain the following comparison result.

\begin{theorem}\label{thm:special-adelic-span}
    Let $L$ have signature $(2,n)$. For $1\leq r\leq n$, special cycles and Heegner cycles span the same rational subspace of codimension-$r$ algebraic cycles on $X_L$.
\end{theorem}
\begin{proof}
    It suffices to express $Z(T,\varphi)$ as a rational linear combination of Heegner cycles. Choose a positive integer $N$ such that
    \[
    \Supp(\varphi)\subset N^{-1}(\widehat{L^\vee})^r.
    \]
    Then $N\bfx\in(L^\vee)^r$ whenever $\bfx\in V^r$ and $\varphi(\bfx)\neq0$.

    Choose representatives $\mathcal R_r$ for
    \[
    \GL_r(\ZZ)\backslash\{A\in M_r(\ZZ):\det A\neq0\}.
    \]
    Every independent tuple $\mathbf z\in(L^\vee)^r$ has a unique expression $\mathbf z=\mathbf yA$, where $A\in\mathcal R_r$ and $\mathbf y$ is a basis of
    \[
    L^\vee\cap\Span_\QQ(z_1,\ldots,z_r).
    \]
    Indeed, changing this basis multiplies $A$ on the left by an element of $\GL_r(\ZZ)$, and fixing $A$ determines $\mathbf y=\mathbf zA^{-1}$. The matrix $A$ records the inclusion of the lattice generated by $\mathbf z$ into its saturation; in particular,
    \[
    |\det A|=\bigl[L^\vee\cap\Span_\QQ(z_1,\ldots,z_r):\ZZ z_1+\cdots+\ZZ z_r\bigr].
    \]

    For $\mathbf z=N\bfx$ with $q(\bfx)=T$, we obtain
    \[
    \bfx=N^{-1}\mathbf yA,\qquad q(\mathbf y)=S_A:=A^{-t}(N^2T)A^{-1},
    \]
    where $A^{-t}=(A^{-1})^t$. To bound these indices, choose a positive integer $c$ such that $cq(\mathbf y)$ has integral entries for all $\mathbf y\in(L^\vee)^r$. If a tuple with Gram matrix $S_A$ exists, then
    \[
    1\leq\det(-cS_A)=\frac{c^rN^{2r}\det(-T)}{(\det A)^2}.
    \]
    Thus $|\det A|$ is bounded, and hence only finitely many representatives $A\in\mathcal R_r$ occur.

    For every $A \in \cR_r$ and $\nu\in A_L^r$, set
    \[
    b_{A,\nu}=\varphi(N^{-1}\mathbf yA),\qquad \mathbf y\in\mathcal P_r(L,\nu),\quad q(\mathbf y)=S_A,
    \]
    and put $b_{A,\nu}=0$ if no such tuple exists. By Lemma~\ref{lem:schwartz-constant-on-saturated-tuples}, this coefficient is independent of $\mathbf y$. The tuples $\bfx$ and $\mathbf y$ have the same orthogonal domain and pointwise stabilizer. Substitution in \eqref{eq:special-cycle-component} therefore gives
    \[
    Z(T,\varphi)=\sum_{A\in\mathcal R_r}\ \sum_{\nu\in A_L^r} b_{A,\nu}\,Z^{\mathrm{prim}}(S_A,\nu),
    \]
    with only finitely many nonzero terms. Lemma~\ref{lem:finite-saturation-inversion} expresses each primitive Heegner cycle as a rational linear combination of Heegner cycles, proving the assertion.
\end{proof}

In particular, this proves Theorem~\ref{thm:comparison-main} in the Chow group.

\section{Auxiliary overlattices}\label{sec:auxiliary}

We construct finitely many overlattices of K3 type whose isotropic induction maps span $\QQ[A_L]$. Their Heegner divisors will generate $\Pic(X_L)_\QQ$ by pullback in Section~\ref{sec:picard}. This construction reduces Theorem~\ref{thm:picard-main} to lattices of K3 type.

\subsection{Isotropic induction}

For a finite nondegenerate quadratic module $(A,\bar{q})$ with polar pairing $\bar{b}$, set $A[p]=\{x\in A:px=0\}$ and $\ell_p(A)=\dim_{\mathbb F_p}A[p]$; put $\ell(A)=\max_p\ell_p(A)$. For an isotropic subgroup $H\subset A$, write $A_H=H^\perp/H$ and define
\[
\uparrow_H^A:\QQ[A_H]\longrightarrow\QQ[A],\qquad \mathbf e_{\mu+H}\longmapsto\sum_{h\in H}\mathbf e_{\mu+h}.
\]
For $a\in\QQ/\ZZ$, let $\QQ[A]_a$ be spanned by the $\mathbf e_\mu$ with $\bar{q}(\mu)=a$. Isotropic induction preserves these subspaces. If $J\subset H$ are isotropic, then 
\begin{equation}\label{eq:induction-transitivity}
    (A_J)_{H/J}\simeq A_H,\qquad \uparrow_H^A=\uparrow_J^A\circ\uparrow_{H/J}^{A_J}.
\end{equation}

\begin{lemma}\label{lem:finiteinduction}
    Fix a prime $p$. Suppose that either $p$ is odd and $\ell_p(A)\geq6$, or $p=2$ and $\ell_2(A)\geq7$, or $p=2$, $\ell_2(A)=6$, and $\bar{q}(A[2])\subset\{0,1/2\}$. Then, for every $a\in\QQ/\ZZ$,
    \begin{equation}\label{eq:finite-induction-spanning}
        \QQ[A]_a=\sum_{\substack{H\subset A\text{ isotropic}\\|H|=p}} \uparrow_H^A\bigl(\QQ[A_H]_a\bigr).
    \end{equation}
\end{lemma}
\begin{proof}
    Write $A=A_p\perp A^p$, where $A_p$ is the $p$-primary component. By \cite[Theorems~4.10 and~4.18]{Mu24}, the hypotheses imply that $\CC[A_p]$ is spanned by isotropic inductions. Indeed, the odd-primary exceptions have length at most five; the $2$-primary exceptions have length at most six, and every exception of length six contains an element of order two with quadratic value $1/4$ or $3/4$.

    By \cite[Lemma~4.4]{Mu24}, we may restrict to isotropic subgroups of order $p$. Since the induction maps are defined over $\QQ$, surjectivity after extension of scalars gives
    \[
    \QQ[A_p]=\sum_{\substack{H\subset A_p\ \mathrm{isotropic}\\ |H|=p}}\uparrow_H^{A_p}\bigl(\QQ[(A_p)_H]\bigr).
    \]
    For each such $H$, we have $A_H=(A_p)_H\perp A^p$ and, under the natural tensor-product identifications,
    \[
    \uparrow_H^A=\uparrow_H^{A_p}\otimes\operatorname{id}_{\QQ[A^p]}.
    \]
    Tensoring with $\QQ[A^p]$ and taking the subspace of quadratic value $a$ gives \eqref{eq:finite-induction-spanning}, since isotropic induction preserves quadratic value.
\end{proof}
\subsection{Isotropic quotients}

Repeated application of Lemma~\ref{lem:finiteinduction} gives the following decomposition.

\begin{proposition}\label{prop:leaves}
    Let $A$ be a finite nondegenerate quadratic module. There exist finitely many isotropic subgroups $H_i\subset A$ such that $C_i=H_i^\perp/H_i$ satisfies
    \begin{enumerate}[label=\textup{(\arabic*)}]
        \item $\ell_p(C_i)\leq5$ for every odd prime $p$, and $\ell_2(C_i)\leq6$;
        \item if $\ell_2(C_i)=6$, then $\bar{q}_{C_i}(z)\in\{1/4,3/4\}$ for some $z\in C_i[2]$.
    \end{enumerate}
    Moreover, for every $a\in\QQ/\ZZ$,
    \begin{equation}\label{eq:terminal-spanning}
        \QQ[A]_a=\sum_i\uparrow_{H_i}^A\bigl(\QQ[C_i]_a\bigr).
    \end{equation}
\end{proposition}
\begin{proof}
    We argue by induction on $|A|$. If $A$ satisfies (1) and (2), take $H_1=0$. Otherwise, Lemma~\ref{lem:finiteinduction} applies at some prime $p$ and gives
    \[
    \QQ[A]_a=\sum_{\substack{J\subset A\text{ isotropic}\\|J|=p}} \uparrow_J^A\bigl(\QQ[A_J]_a\bigr).
    \]
    Since $|A_J|=|A|/p^2$, the induction hypothesis applies to each $A_J$. An isotropic subgroup of $A_J=J^\perp/J$ is of the form $H/J$ for an isotropic subgroup $H\subset A$ containing $J$. For such an $H$, the quotient $(A_J)_{H/J}$ is isomorphic to $A_H$, and \eqref{eq:induction-transitivity} identifies the composite induction map with $\uparrow_H^A$. Substitution gives \eqref{eq:terminal-spanning}. The resulting family is finite and independent of $a$.
\end{proof}

\subsection{Lattices of K3 type}

We now show that the conditions in Proposition~\ref{prop:leaves} imply the required hyperbolic splittings.

\begin{proposition}\label{prop:splitting}
    Suppose $M$ has signature $(2,n)$ with $n\geq10$, and $A_M$ satisfies conditions (1) and (2) of Proposition~\ref{prop:leaves}. Then $M$ is of K3 type.
\end{proposition}
\begin{proof}
    The length bounds give $\ell(A_M)\leq6$. By \cite[Corollary~1.13.5]{Ni79}, the lattice $M$ splits off $U$. Its orthogonal complement has signature $(1,n-1)$, rank $n\geq\ell(A_M)+3$, and the same discriminant form. Applying the same result again, we obtain a splitting of $2U$.

    It remains to show that $M_p$ splits off $3U_p$ for every prime $p$. Put $r=n+2$ and $a_p=r-\ell_p(A_M)$. The local classification in \cite[Corollary~1.9.3]{Ni79} implies that an even $\ZZ_p$-lattice $K$ splits off $U_p$ whenever
    \[
    \operatorname{rk}(K)-\ell_p(A_K)\geq3.
    \]
    Since splitting off $U_p$ preserves the discriminant form, three successive applications give $3U_p$ whenever $a_p\geq7$. For odd $p$, this follows from $a_p\geq r-5\geq7$.

    At $p=2$, the integer $a_2$ is the even rank of the unimodular Jordan component. Thus $a_2\geq8$ unless $r=12$ and $\ell_2(A_M)=6$.

    In this remaining case, choose $z\in A_M[2]$ with $\bar{q}_M(z)\in\{1/4,3/4\}$, by condition (2). Since
    \[
    \bar{b}_M(z,z)=\bar{q}_M(2z)-2\bar{q}_M(z)=1/2,
    \]
    the subgroup $\langle z\rangle$ is nondegenerate and splits off orthogonally from $A_{M_2}$. Thus the last case of \cite[Corollary~1.9.3]{Ni79} applies and gives
    \[
    M_2\simeq U_2^{\oplus t}\perp K_2
    \]
    for some even $\ZZ_2$-lattice $K_2$ with $\operatorname{rk}(K_2)=\ell_2(A_M)=6$. Comparing ranks gives $t=3$.
\end{proof}

Combining Propositions~\ref{prop:leaves} and~\ref{prop:splitting} with the correspondence between isotropic subgroups of $A_L$ and even overlattices of $L$ gives the following corollary.

\begin{corollary}\label{cor:auxiliary-overlattices}
	Let $L$ have signature $(2,n)$ with $n\geq10$. There exist finitely many even overlattices
	\[
    L\subset M_i\subset L^\vee,\qquad 1\leq i\leq m,
	\]
	such that:
	\begin{enumerate}[label=\textup{(\arabic*)}]
        \item Each $M_i$ is of K3 type.
        \item Writing $H_i=M_i/L\subset A_L$ for the corresponding isotropic subgroup, we have
        \[
        \QQ[A_L]_a = \sum_{i=1}^m \uparrow_{H_i}^{A_L}\bigl(\QQ[A_{M_i}]_a\bigr) \qquad(a\in\QQ/\ZZ).
        \]
	\end{enumerate}
\end{corollary}

\section{Picard groups of Baily--Borel compactifications}\label{sec:picard}

In this section, we prove Theorem~\ref{thm:picard-main}. We first show that pullbacks from the auxiliary overlattices constructed in Corollary~\ref{cor:auxiliary-overlattices} generate $\Pic(X_L)_\QQ$, then use finite norms and the Hodge index theorem to deduce Picard rank one for $\overline X_L$.

\subsection{Pullbacks of Heegner divisors}\label{subsec:picard-inputs}

We first recall the following results on Picard groups from \cite{BLMM17} and \cite{HLMY26}.

\begin{theorem}\label{thm:picard-inputs}
    Let $L$ have signature $(2,n)$ with $n\geq5$.
    \begin{enumerate}[label=\textup{(\arabic*)}]
        \item The group $\Pic(X_L)_\QQ$ is generated by the Hodge class and special divisors defined by rational $K_L$-invariant Schwartz functions.
        \item A line bundle algebraically equivalent to zero on $X_L$ has zero class in $\Pic(X_L)_\QQ$.
    \end{enumerate}
    If moreover $L$ is of K3 type and $n>8$, then
    \[
    \Pic(\overline X_L)_\QQ=\QQ\bar\lambda_L.
    \]
\end{theorem}
Statements (1) and (2) follow from \cite[Theorem~3.6 and Corollary~3.8]{BLMM17}, and the last assertion is \cite[Corollary~6.7]{HLMY26}.

We now describe the pullback of Heegner divisors under the maps associated with even overlattices. Let $H\subset A_L$ be isotropic and put $M=L_H$. Recall that
\[
M/L=H,\qquad M^\vee/L=H^\perp,\qquad A_M=H^\perp/H
\]
by \eqref{eq:overlattice-discriminant}. We may use the same component $\mathcal D^+$ for $L$ and $M$.

\begin{lemma}\label{lem:overlattice-morphisms}
    Suppose that $n\geq3$. Then $\widetilde{\rO}^{+}(L)\subset\widetilde{\rO}^{+}(M)$ has finite index, and the induced finite morphism
    \[
    f_H:X_L\longrightarrow X_M
    \]
    extends to a finite surjective morphism
    \[
    \overline f_H:\overline X_L\longrightarrow\overline X_M, \qquad \overline f_H^{-1}(X_M)=X_L.
    \]
    Moreover,
    \[
    f_H^*\lambda_M=\lambda_L, \qquad \overline f_H^*\bar\lambda_M=\bar\lambda_L
    \]
    in the corresponding rational Picard groups.
\end{lemma}
\begin{proof}
    The inclusion follows from \eqref{eq:overlattice-stable-inclusion}. Since $M$ has only finitely many sublattices of index $[M:L]$, the stabilizer of $L$ in $\widetilde{\rO}^{+}(M)$ has finite index. Its action on $A_L$ has kernel $\widetilde{\rO}^{+}(L)$, so $[\widetilde{\rO}^{+}(M):\widetilde{\rO}^{+}(L)]<\infty$.

    The quotient map induces $f_H$ and its finite extension $\overline f_H$ by \cite{BB66}. Since the extension maps boundary to boundary, $\overline f_H^{-1}(X_M)=X_L$.

    The Hodge bundles come from the same equivariant line bundle on $\mathcal D^+$, giving $f_H^*\lambda_M=\lambda_L$. The compactified classes agree by injectivity of restriction to $X_L$, since the boundary has codimension at least two.
\end{proof}

The pullback of a Heegner divisor is given by the coset decomposition of $M$ over $L$.

\begin{proposition}\label{prop:Heegner-pullback}
    Let $H\subset A_L$ be isotropic and put $M=L_H$. For $t<0$ and $\mu\in H^\perp$, we have
    \begin{equation}\label{eq:overlatticepull}
        f_H^*Z_M(t,\mu+H) = \sum_{a\in H}Z_L(t,\mu+a).
    \end{equation}
\end{proposition}
\begin{proof}
    Choose lifts $\widetilde\mu\in M^\vee$ and $\widetilde a\in M$ of $\mu$ and $a\in H=M/L$. The coset decomposition of $\widehat M$ gives
    \[
    \mathbf 1_{\widetilde\mu+\widehat M} = \sum_{a\in H}\mathbf 1_{\widetilde\mu+\widetilde a+\widehat L}.
    \]

    Since $K_L\subset K_M$, \eqref{eq:special-divisor-pullback} gives
    \[
    f_H^*Z_M(t,\mu+H)=Z\bigl(t,\mathbf1_{\widetilde\mu+\widehat M};K_L\bigr)\big|_{X_L}.
    \]
    The characteristic-function identity above and linearity give \eqref{eq:overlatticepull}.
\end{proof}

Combining this formula with isotropic induction gives the desired generation of $\Pic(X_L)_\QQ$.

\begin{proposition}\label{prop:open-Picard-coverage}
    Suppose $n\geq10$. Let $M_1,\ldots,M_m$ be the overlattices provided by Corollary~\ref{cor:auxiliary-overlattices}, and let $f_i:X_L\to X_{M_i}$ be the induced morphisms. Then
    \begin{equation}\label{eq:coverage}
        \Pic(X_L)_{\QQ} = \sum_{i=1}^m f_i^*\Pic(X_{M_i})_{\QQ}.
    \end{equation}
\end{proposition}
\begin{proof}
    Put $H_i=M_i/L$. For $t<0$ and $a=t+\ZZ$, define
    \begin{align*}
        \mathcal Z_{N,t}: \QQ[A_N]_a & \longrightarrow\Pic(X_N)_{\QQ} \\
        \mathbf e_\nu & \longmapsto[Z_N(t,\nu)]
    \end{align*}
    for $N=L,M_1,\ldots,M_m$. Proposition~\ref{prop:Heegner-pullback} gives
    \[
    \mathcal Z_{L,t}\circ\uparrow_{H_i}^{A_L} = f_i^*\circ\mathcal Z_{M_i,t}
    \]
    on $\QQ[A_{M_i}]_a$.

    Since the induction images span $\QQ[A_L]_a$ by Corollary~\ref{cor:auxiliary-overlattices}, every Heegner divisor class lies in the right-hand side of \eqref{eq:coverage}. The same holds for $\lambda_L=f_i^*\lambda_{M_i}$. These classes generate $\Pic(X_L)_\QQ$ by Theorem~\ref{thm:picard-inputs}(1) and Theorem~\ref{thm:comparison-main}, proving the equality.
\end{proof}

\subsection{Cartier divisors and finite morphisms}

Let $X\subset Y$ be a dense open subset of a normal variety, with complement of codimension at least two. For a rational Weil divisor $D=\sum_i a_iD_i$ on $X$, write
\[
\overline D=\sum_i a_i\overline{D_i},
\]
where $\overline{D_i}$ is the closure of $D_i$ in $Y$. 

Let $f:Y\to Y'$ be a finite surjective morphism of normal varieties. The norm construction induces a homomorphism
\[
N_f:\Pic(Y)_\QQ\longrightarrow\Pic(Y')_\QQ;
\]
see \cite[\S~6.5]{EGA}. This map is compatible with proper pushforward on Chow groups, and we therefore also denote it by $f_*$. We will use the following elementary fact.

\begin{lemma}\label{lem:mixed-projection}
    Let $f:Y\to Y'$ be finite and surjective between normal projective $n$-folds with $n\geq3$, and let
    $X'\subset Y'$ be open with $\operatorname{codim}(Y'\setminus X',Y')\geq3$. Put $X=f^{-1}(X')$ and $f_X=f|_X$. For a rational Cartier divisor $\Delta$ on $X'$, set $D=\overline\Delta$ and $E=\overline{f_X^*\Delta}$. Then, for $\gamma\in\Pic(Y)_\QQ$ and $h'\in\Pic(Y')_\QQ$, writing $h=f^*h'$, we have
    \[
    \deg(\gamma\cdot E\cdot h^{n-2})=\deg\bigl((f_*\gamma)\cdot D\cdot(h')^{n-2}\bigr).
    \]
\end{lemma}
\begin{proof}
    On $X'$, the projection formula gives
    \[
    \bigl.f_*(\gamma\cdot E)\bigr|_{X'}=\Delta\cdot f_{X,*}(\gamma|_X)=(f_*\gamma)|_{X'}\cdot\Delta.
    \]
    By localization and the codimension assumption, restriction
    \[
    \CH_{n-2}(Y')_\QQ\longrightarrow\CH_{n-2}(X')_\QQ
    \]
    is injective. Hence
    \[
    f_*(\gamma\cdot E)=(f_*\gamma)\cdot D\qquad\text{in }\CH_{n-2}(Y')_\QQ.
    \]
    The assertion then follows from intersecting with $(h')^{n-2}$ and applying the projection formula.
\end{proof}

\subsection{A criterion for Picard number one}

For a projective variety $Y$, let $N^1(Y)_\QQ$ denote the space of rational Cartier divisors modulo numerical equivalence. We combine Lemma~\ref{lem:mixed-projection} with the Hodge index theorem to obtain the following criterion.

\begin{proposition}\label{prop:joint}
    Let $Y,Y_1,\ldots,Y_m$ be normal projective varieties of dimension $n\geq3$, with $m\geq1$, and let
    \[
    f_i:Y\longrightarrow Y_i,\qquad 1\leq i\leq m,
    \]
    be finite surjective morphisms. Suppose that there are dense open subsets $X\subset Y$ and $X_i\subset Y_i$ such that
    \[
    f_i^{-1}(X_i)=X,\qquad \dim(Y_i\setminus X_i)\leq n-3.
    \]
    Write $f_{i, X}: X\to X_i$ for the restrictions of $f_i$ to $X$. Assume that:
    \begin{enumerate}[label=\textup{(\arabic*)}]
        \item There are ample Cartier classes $h_i\in\Pic(Y_i)_{\QQ}$ and a class $h\in\Pic(Y)_{\QQ}$ satisfying
        \[
        f_i^*h_i=h \qquad\text{for every }i.
        \]

        \item The Picard groups satisfy
        \[
        \Pic(X)_{\QQ}=\sum_{i=1}^m f_{i,X}^*\Pic(X_i)_{\QQ}, \qquad \Pic(Y_i)_{\QQ}=\QQ h_i \quad\text{for every }i.
        \]

        \item The natural map
        \[
        \Pic(Y)_{\QQ}\longrightarrow N^1(Y)_{\QQ}
        \]
        is injective.
    \end{enumerate}
    Then
    \[
    \Pic(Y)_{\QQ}=\QQ h.
    \]
\end{proposition}
\begin{proof}
    The class $h$ is ample, since it is the pullback of an ample class under a finite morphism. Let $\alpha\in\Pic(Y)_{\QQ}$, and set
    \[
    a=\frac{\deg(\alpha\cdot h^{n-1})}{\deg(h^n)}, \qquad \beta=\alpha-ah.
    \]
    Then $\deg(\beta\cdot h^{n-1})=0$.

    Since $\Pic(Y_i)_\QQ=\QQ h_i$, the class $(f_i)_*\beta$ is a rational multiple of $h_i$. The projection formula gives
    \[
    \deg\bigl((f_i)_*\beta\cdot h_i^{n-1}\bigr)=\deg(\beta\cdot h^{n-1})=0.
    \]
    Since $\deg(h_i^n)>0$, we obtain
    \[
    (f_i)_*\beta=0\qquad\text{in }\Pic(Y_i)_\QQ\quad\text{for every }i.
    \]

    The generation statement in assumption~(2) gives classes $\beta_i\in\Pic(X_i)_\QQ$ such that
    \[
    \beta|_X=\sum_{i=1}^m f_{i,X}^*\beta_i.
    \]
    Represent $\beta_i$ by a rational Cartier divisor $\Delta_i$ on $X_i$, and let
    \[
    D_i=\overline{\Delta_i},\qquad E_i=\overline{f_{i,X}^*\Delta_i}
    \]
    denote their Weil closures in $Y_i$ and $Y$, respectively. Since $f_i$ is finite,
    \[
    \dim(Y\setminus X)=\dim(Y_i\setminus X_i) \leq n-3.
    \]
    Restriction identifies $\Cl(Y)_{\QQ}$ with $\Cl(X)_{\QQ}$, and we obtain
    \[
    [\beta]=\sum_{i=1}^m[E_i] \qquad\text{in }\Cl(Y)_{\QQ}.
    \]

    Applying Lemma~\ref{lem:mixed-projection} to each $f_i$ with $\gamma=\beta$, we find
    \begin{align*}
        \deg(\beta^2\cdot h^{n-2}) & = \sum_{i=1}^m \deg(\beta\cdot E_i\cdot h^{n-2}) \\
        & = \sum_{i=1}^m \deg\bigl((f_i)_*\beta\cdot D_i\cdot h_i^{n-2}\bigr) \\
        & = 0.
    \end{align*}
    By \cite[Theorem~6.3]{Kleiman05}, $\beta$ is numerically trivial. Assumption~(3) gives $\beta=0$ in $\Pic(Y)_{\QQ}$. Hence $\alpha=ah$, proving the proposition.
\end{proof}

\subsection{Proof of Theorem~\ref{thm:picard-main}}\label{subsec:overlattice-rank-one}

For $n\geq5$, we have
\begin{equation}\label{eq:BB-Picard-injections}
    \Pic(\overline X_L)_\QQ\hookrightarrow\Pic(X_L)_\QQ,
    \qquad
    \Pic(\overline X_L)_\QQ\hookrightarrow N^1(\overline X_L)_\QQ.
\end{equation}
The first injection follows from normality and the fact that the boundary has codimension at least two. For the second, a numerically trivial line bundle on $\overline X_L$ has an algebraically trivial positive tensor power by \cite[Theorem~6.3]{Kleiman05}. Its restriction therefore vanishes in $\Pic(X_L)_\QQ$ by Theorem~\ref{thm:picard-inputs}(2), and the first injection gives the desired conclusion.

We then apply Proposition~\ref{prop:joint} to the maps associated with the auxiliary overlattices.

\begin{theorem}\label{thm:rank10}
    Let $L$ be a lattice of signature $(2,n)$ with $n\geq10$. Then
    \[
    \Pic(\overline X_L)_{\QQ} = \QQ\bar\lambda_L.
    \]
\end{theorem}

\begin{proof}
    Choose $M_i$ as in Corollary~\ref{cor:auxiliary-overlattices}. By Lemma~\ref{lem:overlattice-morphisms}, the maps $f_i:X_L\to X_{M_i}$ extend to finite surjective morphisms
    \[
    \overline f_i:\overline X_L\longrightarrow\overline X_{M_i},\qquad \overline f_i^{-1}(X_{M_i})=X_L,\qquad \overline f_i^*\bar\lambda_{M_i}=\bar\lambda_L.
    \]
    Each $M_i$ is of K3 type, so Theorem~\ref{thm:picard-inputs} gives
    \[
    \Pic(\overline X_{M_i})_\QQ=\QQ\bar\lambda_{M_i}.
    \]
    Proposition~\ref{prop:open-Picard-coverage} gives generation by pullbacks on $X_L$, and \eqref{eq:BB-Picard-injections} gives injectivity into $N^1(\overline X_L)_\QQ$. Since the boundary has dimension at most one and the Hodge classes are ample, Proposition~\ref{prop:joint} applies.
\end{proof}

For an arbitrary congruence subgroup, we use the sublattice $aL\subset L$.

\begin{proposition}\label{prop:scalar-refinement}
    Let $L$ have signature $(2,n)$, and let
    $\Gamma\subseteq\rO^+(L)$ be a congruence subgroup.
    There exists an integer $a\geq1$ such that
    \[
    \widetilde{\rO}^{+}(aL)\subseteq\Gamma.
    \]
\end{proposition}
\begin{proof}
    Choose an integer $m\geq1$ such that $\Gamma$ contains the principal congruence subgroup
    \[
    \ker\bigl(\rO^+(L)\longrightarrow \operatorname{Aut}(L/mL)\bigr),
    \]
    and choose $a\geq1$ with $m\mid a^2$.

    Note that $\rO^+(aL)=\rO^+(L)$. If $g\in\widetilde{\rO}^{+}(aL)$, triviality on
    \[
    A_{aL}=a^{-1}L^\vee/aL
    \]
    applied to $y/a$, for $y\in L^\vee$, gives
    \[
    (g-1)y\in a^2L.
    \]
    Since $L\subseteq L^\vee$, the element $g$ acts trivially on $L/a^2L$, hence also on $L/mL$. Therefore $g\in\Gamma$.
\end{proof}

We can now prove the theorem for $X_\Gamma$.

\begin{proof}[Proof of Theorem~\ref{thm:picard-main}]
    For a congruence subgroup $\Gamma\subseteq\rO^+(L)$, Proposition~\ref{prop:scalar-refinement} gives an integer $a$ such that $\widetilde{\rO}^{+}(aL)\subseteq\Gamma$. The finite map
    \[
    g:\overline X_{aL}\longrightarrow\overline X_\Gamma
    \]
    satisfies $g^*\bar\lambda_\Gamma=\bar\lambda_{aL}$. The lattice $aL$ is even of the same signature, so Theorem~\ref{thm:rank10} gives $\Pic(\overline X_{aL})_\QQ=\QQ\bar\lambda_{aL}$. For $\alpha\in\Pic(\overline X_\Gamma)_\QQ$, write $g^*\alpha=c\bar\lambda_{aL}$. Injectivity of finite pullback implies $\alpha=c\bar\lambda_\Gamma$, proving the theorem.
\end{proof}

\begin{corollary}
    Let $L$ have signature $(2,n)$ with $n\geq10$. The Hodge class and Heegner divisor classes generate $\Pic(X_L)_\QQ$. A class in $\Pic(X_L)_\QQ$ extends to $\overline X_L$ if and only if it is a rational multiple of $\lambda_L$.
\end{corollary}
\begin{proof}
    Generation follows from Theorem~\ref{thm:picard-inputs}(1) and Theorem~\ref{thm:comparison-main}. By \eqref{eq:BB-Picard-injections}, restriction from the compactification is injective, and Theorem~\ref{thm:picard-main} identifies its image with $\QQ\lambda_L$.
\end{proof}

\begin{remark}
    The bound $n\geq10$ is used in Proposition~\ref{prop:splitting}. In dimension nine, the length bounds of Proposition~\ref{prop:leaves} do not imply a splitting of $3U_p$ at every prime. For example, consider
    \[
    T=U^{\oplus2}\perp V\perp V(2)^{\oplus2}\perp\langle-4\rangle,
    \]
    where $V$ is the rank $2$ negative-definite even lattice with Gram matrix 
	\[
    \begin{pmatrix}
        -2&-1\\
        -1&-2
    \end{pmatrix}.
	\]
    At $p=2$, we have
	\[
    T_2\simeq U_2^{\oplus2}\perp V_2\perp K_2, \qquad K_2=\bigl(V(2)^{\oplus2}\perp\langle-4\rangle\bigr)_2.
	\]
	Moreover,
	\[
    \operatorname{rk}_{\ZZ_2}(K_2)=\ell_2(A_T)=5.
	\]
	Thus the decomposition is the normal form in the second case of \cite[Corollary~1.9.3]{Ni79}. By uniqueness of normal form, $T_2$ cannot admit a decomposition $U_2^{\oplus3}\perp K_2'$. 
\end{remark}

\bibliographystyle{amsalpha}
\bibliography{main}

\end{document}